\documentclass[12pt]{article}
\usepackage[margin=1in]{geometry}
\usepackage{amsmath,amsthm,amssymb}
\usepackage{verbatim}
\usepackage{tikz}
\usepackage{mathtools}

\tikzstyle{vertex}=[draw,thick,fill=white,circle,inner sep=2pt]
\tikzstyle{tiny_vertex}=[draw,thick,fill=white,circle,inner sep=1.5pt]
\tikzstyle{full}=[draw,thick,fill=black,circle,inner sep=2pt]
\tikzstyle{empty}=[draw,color=black!40!white,thick,fill=white,circle,inner sep=2pt]

\newtheorem{thm}{Theorem}[section]
\newtheorem{cor}[thm]{Corollary}
\newtheorem{lem}[thm]{Lemma}

\theoremstyle{definition}

\newtheorem{defn}[thm]{Definition}
\newtheorem{ex}[thm]{Example}

\allowdisplaybreaks

\DeclareMathOperator{\dist}{dist}
\DeclareMathOperator{\spec}{spec}

\DeclareMathOperator{\diag}{diag}

\newcommand{\DL}{\mathcal{D}^L}

\newcommand{\T}[1]{~^T\!#1}

\title{Cospectral constructions for several graph matrices using cousin vertices}
\author{Kate Lorenzen\footnote{Iowa State University, Ames, IA, USA \texttt{lorenkj@iastate.edu}}}
\date{\empty}

\begin{document}
\maketitle

\begin{abstract}
Graphs can be associated with a matrix according to some rule and we can find the spectrum of a graph with respect to that matrix. 
Two graphs are cospectral if they have the same spectrum.  Constructions of cospectral graphs help us establish patterns about structural information not preserved by the spectrum. We generalize a construction for cospectral graphs previously given for the distance Laplacian matrix to a larger family of graphs. 
In addition, we show that with appropriate assumptions this generalized construction extends to the adjacency matrix, combinatorial Laplacian matrix, signless Laplacian matrix, normalized Laplacian matrix, and distance matrix. 
\end{abstract}



\section{Introduction}
Let $G=(V,E)$ be a graph and $M$ a matrix associated with the graph. 
Then the spectrum (multiset of eigenvalues) of $M$, $\spec_M(G)$, is referred to as the spectrum of $G$ with respect to $M$. 
If $M$ is clear, we will say the spectrum of $G$, denoted $\spec(G)$. 
If two graphs $G_1,G_2$ share the same spectrum with respect to $M$, i.e. $\spec_M(G_1)=\spec_M(G_2)$, then we say that $G_1,G_2$ are cospectral graphs. 
When $G_1$ is not isomorphic to $G_2$, this is an interesting relationship between the two graphs because it gives us information into what structural properties of a graph are not preserved by $M$.
There are many possible choices for $M$, and each gives us different insight into the graph. 

Given $M$, examples of cospectral graphs are (usually) easy to find for a small number of vertices by exhaustive search. To find families of cospectral graphs on a large number of vertices we need constructions. These constructions help us understand how information about the structural proprieties of a graph are not determined by the spectrum and demonstrate weaknesses of a matrix. 

Cospectral constructions have been studied for the adjacency matrix \cite{GM}, combinatorial Laplacian \cite{vDH03, HS04}, signless Laplacian \cite{HS04}, normalized Laplacian \cite{BG11}, distance matrix \cite{heysse}, and to a lesser extent distance Laplacian matrix \cite{AH13,grwc}.
The adjacency matrix $A$ has a $1$ in the $i,j$ entry if there is an edge between vertices $i,j$ and $0$ otherwise. 
The combinatorial Laplacian is defined as $L=D-A$ where $D$ is the diagonal matrix with the degrees of the vertices down the diagonal. The signless Laplacian is defined as $|L|=D+A$. The normalized Laplacian is defined as $\mathcal{L}=D^{-1/2}LD^{-1/2}$. These matrices off diagonal zero-nonzero pattern is that of the adjacency matrix so we will refer to these matrices as adjacency matrices. 

The distance matrix $\mathcal{D}$ has entries $\mathcal{D}_{i,j}=\dist(i,j)$ where the distance is the length of the shortest path between vertex $i$ and vertex $j$. The distance Laplacian $\DL=T-\mathcal{D}$ where $T$ is the diagonal matrix with the transmission of the vertices (sum of the distances to a particular vertex) down the diagonal. When dealing with distance matrices, the graph is assumed to be connected. 

The most well known cospectral construction for graphs is Godsil-McKay switching for the adjacency matrix \cite{GM}. 
The construction is a special case of Seidel switching. 
Seidel switching on a graph $G$ with switching set $S$ produces a new graph $G'$ on the vertices of $G$ by keeping edges with both of the endpoints in either $S$ or $G\backslash S$, and adds an edge between vertices in $S$ and $G \backslash S$ if only if that edge was not in $G$. An example of this switching is shown in Figure \ref{fig:GM-Switching}. 

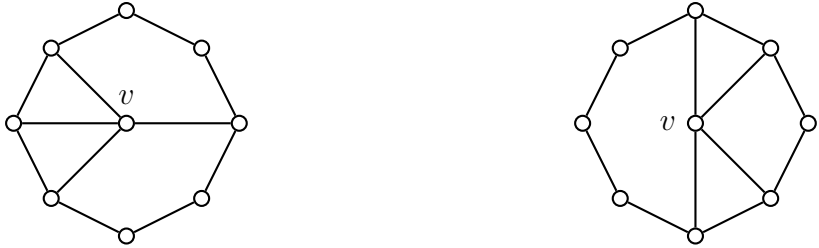
\begin{figure}[h]
    \centering
    \begin{tikzpicture}
    \node[vertex] (b1) at (0,0) {};
\node[vertex] (b2) at (0.5,1) {};
\node[vertex] (b3) at (1.5,1.5) {};
\node[vertex] (b4) at (2.5,1) {};
\node[vertex] (b5) at (3,0) {};
\node[vertex] (b6) at (0.5,-1) {};
\node[vertex] (b7) at (1.5,-1.5) {};
\node[vertex] (b8) at (2.5,-1) {};

\node[vertex] (a1) at (1.5,0)[label=above:$v$] {};

\draw[thick] (b1)--(b2)--(b3)--(b4)--(b5)--(b8)--(b7)--(b6)--(b1);
\draw[thick] (b1)--(a1)--(b2);
\draw[thick] (b6)--(a1)--(b5);
\end{tikzpicture}
\hspace{1.6in}
\begin{tikzpicture}
\node[vertex] (b1) at (0,0) {};
\node[vertex] (b2) at (0.5,1) {};
\node[vertex] (b3) at (1.5,1.5) {};
\node[vertex] (b4) at (2.5,1) {};
\node[vertex] (b5) at (3,0) {};
\node[vertex] (b6) at (0.5,-1) {};
\node[vertex] (b7) at (1.5,-1.5) {};
\node[vertex] (b8) at (2.5,-1) {};

\node[vertex] (a1) at (1.5,0)[label=left:$v$] {};

\draw[thick] (b1)--(b2)--(b3)--(b4)--(b5)--(b8)--(b7)--(b6)--(b1);
\draw[thick] (b3)--(a1)--(b4);
\draw[thick] (b7)--(a1)--(b8);
    \end{tikzpicture}
    \caption{Two graphs that are cospectral for the adjacency matrix by Godsil-McKay switching about $S=\{v\}$.}
    \label{fig:GM-Switching}
\end{figure}

Godsil-McKay switching puts conditions on the graph $G$ and the switching set to create a pair of cospectral graphs. Their proof consists of showing that a matrix $\mathcal{C}=\diag(\frac1k J-I,I)$ is a similarity matrix for the adjacency matrix of the two cospectral graphs. 
Haemers and Spence extended the construction to $L$ and $|L|$ by introducing the GM*-property and its relaxation for graphs \cite{HS04}. 
The GM*-property is a set of sufficient graph conditions such that $\mathcal{C}$ is a similarity matrix for matrices of the form $M=\alpha A + \beta I + \gamma D$. The relaxation of the GM*-property relaxes the graph conditions if the matrix has constant row sums. This extension is one of the most well known cospectral construction for the Laplacian and signless Laplacian matrices. 

This construction and others for well studied matrices can be thought of as a \emph{switch} of a set of edges \cite{AH12,GM, HS04,heysse,WQH19}. In this paper, we present a construction that \emph{swaps} a set of edges. The difference between these two concepts is in a switch exactly one endpoint of an edge changes while in a swap both endpoints of an edge change. 

Our result broadens a construction for the distance Laplacian matrix and extend it to the distance matrix, matrices of the form $M=\alpha A + \beta I + \gamma D$, and the normalized Laplacian. We do this by demonstrating graph conditions such that $\diag(\frac1k J -\hat{I},I)$ is a similarity matrix where $\hat{I}$  is the matrix with ones along the anti-diagonal and zeros elsewhere. This is the only known cospectral construction for both distance and adjacency matrices (when the diameter of the graph is greater than two). 

\begin{figure}[h]
    \centering
    \begin{tabular}{ccc}
       \begin{tikzpicture}[scale=0.8]
    \node[vertex] (a1) at (0,0) {};
    \node[vertex] (a2) at (1,0) {};
    \node[vertex] (a3) at (2,0) {};
    \node[vertex] (a4) at (1.5,1) {};
    
    \node[vertex] (b1) at (-1, 1)[label=left:$u_1$] {};
    \node[vertex] (b2) at (-1,0)[label=left:$u_2$] {};
    \node[vertex] (b3) at (-1,-1)[label=left:$u_3$] {};
    
    \node[vertex] (c1) at (3, 1)[label=right:$w_1$] {};
    \node[vertex] (c2) at (3,0)[label=right:$w_2$] {};
    \node[vertex] (c3) at (3,-1)[label=right:$w_3$] {};

    \draw[thick] (a1)--(a2)--(a3)--(a4)--(a2)
    (b1)--(a1)--(b2)
    (a1)--(b3)
    (c1)--(a3)--(c2)
    (a3)--(c3);
    \end{tikzpicture} &\begin{tikzpicture}[scale=0.8]
    \node[vertex] (a1) at (0,0) {};
    \node[vertex] (a2) at (1,0) {};
    \node[vertex] (a3) at (2,0) {};
    \node[vertex] (a4) at (1.5,1) {};
    
    \node[vertex] (b1) at (-1, 1)[label=left:$u_1$] {};
    \node[vertex] (b2) at (-1,0)[label=left:$u_2$] {};
    \node[vertex] (b3) at (-1,-1)[label=left:$u_3$] {};
    
    \node[vertex] (c1) at (3, 1)[label=right:$w_1$] {};
    \node[vertex] (c2) at (3,0)[label=right:$w_2$] {};
    \node[vertex] (c3) at (3,-1)[label=right:$w_3$] {};

    \draw[thick] (a1)--(a2)--(a3)--(a4)--(a2)
    (b1)--(a1)--(b2)
    (a1)--(b3)
    (c1)--(a3)--(c2)
    (a3)--(c3)
    (b1)--(b2)--(b3);
    \end{tikzpicture} &  \begin{tikzpicture}[scale=0.8]
    \node[vertex] (a1) at (0,0) {};
    \node[vertex] (a2) at (1,0) {};
    \node[vertex] (a3) at (2,0) {};
    \node[vertex] (a4) at (1.5,1) {};
    
    \node[vertex] (b1) at (-1, 1)[label=left:$u_1$] {};
    \node[vertex] (b2) at (-1,0)[label=left:$u_2$] {};
    \node[vertex] (b3) at (-1,-1)[label=left:$u_3$] {};
    
    \node[vertex] (c1) at (3, 1)[label=right:$w_1$] {};
    \node[vertex] (c2) at (3,0)[label=right:$w_2$] {};
    \node[vertex] (c3) at (3,-1)[label=right:$w_3$] {};

    \draw[thick] (a1)--(a2)--(a3)--(a4)--(a2)
    (b1)--(a1)--(b2)
    (a1)--(b3)
    (c1)--(a3)--(c2)
    (a3)--(c3)
    (c1)--(c2)--(c3);
    \end{tikzpicture}\\
    (a)& (b) & (c)
    \end{tabular}
    
    \caption{(a) A graph with isolated twin set $\{u_1,u_2,u_3\}$ and $\{w_1,w_2,w_3\}$ that are the same size and with degree $1$. By adding $2$ edges into each set, we produced cospectral graphs ((b) and (c)) for the Laplacian matrix. }
    \label{fig:LaplaFolklore}
\end{figure}
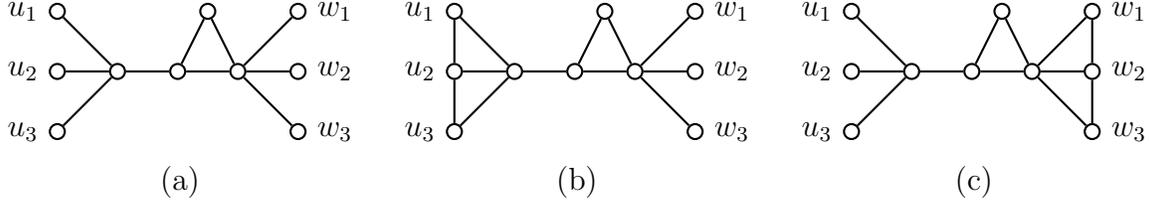

This construction exploits well known spectral properties of twin vertices in both distance and adjacency matrices. \emph{(Isolated) Twin vertices} are two vertices in a graph that have the same neighborhood (and thus are not connected to each other). In a distance or adjacency matrix representation of a graph, the columns (and rows) corresponding to twin vertices, $v_1,v_2$, have the same entries in each row (and column) corresponding to the other vertices in the graph. In other words, $M_{v_1,u}=M_{v_2,u}$ for all $u \in V \backslash \{v_1,v_2\}$. This allows $[1,-1, 0, \cdots, 0]^T$ to be an eigenvector of our matrix. Since these matrices are real symmetric matrices, it follows that our remaining eigenvectors can be chosen so the first two entries are the same. 

This leads to a very natural cospectral construction which is formally given for the combinatorial Laplacian matrix in \cite{das} and the distance Laplacian matrix in \cite{grwc}. For the combinatorial Laplacian matrix, the construction says that if we have two sets of isolated twin vertices of the same size and have the same degree, then adding $k$ edges into one set is cospectral to adding $k$ edges into the other set. In other words, these $k$ edges can be \emph{swapped} from one twin set to another. An example of this construction is shown in Figure \ref{fig:LaplaFolklore}.

We will show that the twin condition can be relaxed for the combinatorial Laplacian matrix when constructing cospectral graphs and can be broadened to include all adjacency and distance matrices. This relaxation is an extension of a cospectral construction for the distance Laplacian given in \cite{grwc} and an example is shown in Figure \ref{fig:coTransCous}. 

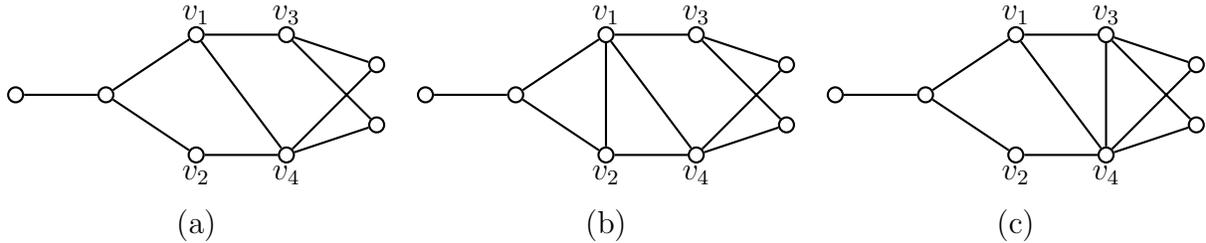
\begin{figure}[h]
    \centering
    \begin{tabular}{ccc}
        \begin{tikzpicture}[scale=0.4]

\node[vertex]  (v1) at (3,0) {};
\node[vertex]  (v2) at (6,2) {}; \node [above] at (6,2) {$v_1$};
\node[vertex] (v3) at (6,-2) {}; \node [below] at (6, -2) {$v_2$};
\node[vertex]  (v4) at (9,2) {}; \node [above] at (9, 2) {$v_3$};
\node[vertex]  (v5) at (9,-2) {}; \node [below] at (9,-2) {$v_4$};

\node[vertex]  (v0) at (0,0) {};
\node[vertex]  (v6) at (12,-1) {};
\node[vertex] (v7) at (12,1) {};

\draw[thick] (v0)--(v1)--(v2)--(v5)--(v3)--(v1)
(v2)--(v4)--(v6)--(v5)--(v7)--(v4);
\end{tikzpicture} & 
\begin{tikzpicture}[scale=0.4]

\node[vertex]  (v1) at (3,0) {};
\node[vertex]  (v2) at (6,2) {}; \node [above] at (6,2) {$v_1$};
\node[vertex] (v3) at (6,-2) {}; \node [below] at (6, -2) {$v_2$};
\node[vertex]  (v4) at (9,2) {}; \node [above] at (9, 2) {$v_3$};
\node[vertex]  (v5) at (9,-2) {}; \node [below] at (9,-2) {$v_4$};

\node[vertex]  (v0) at (0,0) {};
\node[vertex]  (v6) at (12,-1) {};
\node[vertex] (v7) at (12,1) {};

\draw[thick] (v0)--(v1)--(v2)--(v5)--(v3)--(v1)
(v2)--(v4)--(v6)--(v5)--(v7)--(v4)
(v2)--(v3);
\end{tikzpicture}&
\begin{tikzpicture}[scale=0.4]

\node[vertex]  (v1) at (3,0) {};
\node[vertex]  (v2) at (6,2) {}; \node [above] at (6,2) {$v_1$};
\node[vertex] (v3) at (6,-2) {}; \node [below] at (6, -2) {$v_2$};
\node[vertex]  (v4) at (9,2) {}; \node [above] at (9, 2) {$v_3$};
\node[vertex]  (v5) at (9,-2) {}; \node [below] at (9,-2) {$v_4$};

\node[vertex]  (v0) at (0,0) {};
\node[vertex]  (v6) at (12,-1) {};
\node[vertex] (v7) at (12,1) {};

\draw[thick] (v0)--(v1)--(v2)--(v5)--(v3)--(v1)
(v2)--(v4)--(v6)--(v5)--(v7)--(v4)
(v4)--(v5);
\end{tikzpicture}\\
     (a) & (b)    & (c)
    \end{tabular}
    \caption{(a) A graph $G$ with co-transmission cousin pair $\{v_1,v_2\}, \{v_3,v_4\}$. The graphs shown in (b) and (c) are the two graphs formed by \emph{swapping} edges to create a pair of cospectral graphs using the construction given in \cite{grwc}. }
    \label{fig:coTransCous}
\end{figure}

We will describe this construction in Section \ref{sec:construc} and give generalized definitions of sets of twin vertices called \emph{cousins}. These definitions allow us to show that a matrix $\mathcal{S}=\diag(\frac1k J- \hat{I},I)$ is a similarity matrix between graphs where we perform a \emph{swap} on the edges between the cousin vertices.

This construction method demonstrates a loss of information about the graph structure from many different matrix representations. To show this construction, we will first introduce \emph{cousin} vertices and prove some simple results about \emph{gluing} edges within cousins. Then we will prove some linear algebra results about real symmetric matrices to show that our graph construction preserves the spectrum of adjacency and distance matrices.

\section{Extension of Cousin Cospectral Construction} \label{sec:construc}

The cospectral construction given in \cite{grwc} which we wish to extend starts with a set of vertices with special proprieties called \emph{cousins}.

\begin{defn} \cite{grwc} \label{def:coTransCous2}
In a graph $G$, the set of vertices $C=\{v_1,v_2\} \cup \{v_3, v_4\}$ are called a set of \emph{co-transmission cousins} if 
\begin{enumerate}
    \item For all $u \in V(G) \backslash C$, $\dist(v_1,u)=\dist(v_2,u)$ and $\dist(v_3,u)=\dist(v_4,u)$;
    \item $\sum_{u\in V(G)\backslash C} \dist(v_1,u)= \sum_{u\in V(G)\backslash C} \dist(v_3,u)$.
\end{enumerate}
\end{defn}

We can think of them as two sets of twin vertices with the same transmission if we ignore adjacencies between our special four vertices. The construction of cospectral graphs for the distance Laplacian \emph{swaps} $K_2$ with $\overline{K_2}$ between $\{v_1,v_2\}$ and $\{v_3, v_4\}$ with some conditions to create a pair of cospectral graphs. An example of this construction is shown in Figure \ref{fig:coTransCous}.

Our construction extends \emph{swapping} of edges from a pair of two vertex sets to a pair of $m$ vertex sets. Therefore, we will first extend the definition of cousins. 

\begin{defn}
In a graph $G$, the pair of sets of vertices $U$ and $W$ are called \emph{cousins} if
\begin{enumerate}
    \item $|U|=|W|=m$;
    \item $\dist(u_i,v)=\dist(u_j,v)$ and $\dist(w_i,v)=\dist(w_j,v)$ for all $u_i,u_j \in U$, all $w_i,w_j \in W$, and all $v \in V(G) \backslash \{U \cup W\}$.
    
    Additionally, we would call a pair \emph{co-transmission cousins} if 
    \item $\displaystyle \sum_{v\in V(G)\backslash \{U \cup W\}}\dist(u_i,v)= \sum_{v\in V(G)\backslash \{U \cup W\}} \dist(w_j,v)$ for all $u_i \in U$ and all $w_j \in W$. 
\end{enumerate}
\end{defn}

 These definitions about pairs of sets of vertices will be used for our construction for the distance and distance Laplacian matrices. The next definitions about pairs of sets of vertices will be used for our construction for the adjacency, combinatorial Laplacian, and signless Laplacian matrices. 

\begin{defn}
In a graph $G$, the pair of sets of vertices $U$ and $W$ are called  \emph{relaxed cousins} if
\begin{enumerate}
    \item $|U|=|W|=m$;
    \item $u_i \sim v$ if and only if $u_j\sim v$ and $w_i \sim v$ if and only if $w_j\sim v$ for all $u_i,u_j \in U$, all $w_i, w_j \in W$, and all $v \in V(G) \backslash \{U \cup W\}$. 
    
Additionally, we would call a pair \emph{co-degree (relaxed) cousins} if 

\item $|N(u_i)\backslash W|=|N(w_j)\backslash U|$ for all $u_i \in U$ and all $w_j \in W$. 
\end{enumerate}
\end{defn}

We will note that a pair of relaxed cousins encompasses a pair of (co-transmission) cousins. Additionally, if there are no edges between sets $U,W$, then these must be a pair of sets of twin vertices. 

For our cospectral constructions, we will start with a base graph and \emph{glue} in two graphs two different ways creating cospectral graphs. This operation can formally be defined using maps and the following example demonstrates \emph{gluing} by using maps. 

\begin{ex}
Let $G$ be an empty graph on four vertices and let $H$ be the complete graph on three vertices. So $V(G)=\{v_1, v_2, v_3, v_4\}$ with $E(G)=\emptyset$ and $V(H)=\{u_1, u_2, u_3\}$ with $E(H)=\{u_1u_2, u_2u_3, u_3u_1\}$. 

Let $\phi:V(H)\to V(G)$ be an injective function such that $\phi(u_i)=v_i$. $\phi$ will be our \emph{gluing} map. A new graph \begin{align*}
    G'&=G+\phi(E(H))\\
    &=G+ \phi(\{u_1u_2, u_2u_3, u_3u_1\})\\
    &=G+ \{\phi(u_1u_2), \phi(u_2u_3), \phi(u_3u_1) \}\\
    &=G + \{v_1v_2, v_2v_3, v_3v_1\}
\end{align*} is the graph $K_3$ union with an isolated vertex.  We would say that $G'$ is $G$ with $H$ glued into $G$ with respect to $\phi$. 
\end{ex} 

An interesting property about cousins is that gluing into $U$ or $W$ does not change the length of a shortest path between two vertices $u,v$ where $v$ is not in $U$ or $W$. 

\begin{lem} \label{notInPath}
Let $G$ be a graph with cousins $V_1,V_2$ on $m$ vertices. Let $H_1,H_2$ be any two graphs on $m$ vertices and $\phi_{i,j}$ be a bijective mapping from $H_i$ to $V_j$ for $i,j \in\{1,2\}$. 
Let $G_1=G+\phi_{1,1}(E(H_1))+\phi_{2,2}(E(H_2))$ and $G_2=G+\phi_{2,1}(E(H_2))+\phi_{1,2}(E(H_1))$.

Then for all $u \not \in V_1 \cup V_2$, $v \in V(G)$, $\dist_{G_1}(u,v)= \dist_{G_2}(u,v)$.
\end{lem}

\begin{proof}
Let $u \not \in V_1 \cup V_2$ and $v \in V(G)$. If a shortest path in $G_1$ does not contain an edge with both its endpoints in $V_1$ nor $V_2$ and a shortest path in $G_2$ does not contain an edge with both its endpoints in $V_1$ nor $V_2$, then the distances are the same, i.e. $\dist_{G_1}(u,v)=\dist_{G_2}(u,v)$ since $G_1$ and $G_2$ only differ by edges that have both of their endpoints in $V_1$ or $V_2$.

Let $\mathcal{P}$ be a shortest path in $G_1$ between $u,v$ that contains, without loss of generality, edge $(w_1,w_2)$ where $w_1,w_2 \in V_1$. We know for any subpath $\mathcal{P}' \subseteq \mathcal{P}$ that starts at vertex $x$ and ends at vertex $y$, then $\mathcal{P}'$ is a shortest path between $x$ and $y$. 

This implies that a shortest path between $u$ and $w_2$ uses edge $(w_1,w_2)$. 
Moreover, $\dist(u,w_1) +1=\dist (u,w_2)$ but every $u \not \in V_1 \cup V_2$ must be equidistant from vertices in $V_1$ (and $V_2$) since they are cousins. Therefore, there is no shortest path $\mathcal{P}$ that contains an edge with both of its endpoints in $V_1$. There is an analogous argument for $V_2$ and $G_2$. 

Thus the distance between $u,v$ is preserved from $G_1$ to $G_2$. 
\end{proof}

An additional interesting property of gluing is when we glue into a bipartite graph $G=(A \cup B, E)$ that has an automorphism $\pi$ where $\pi(A)=B$ and $\pi^2=id$, then gluing any graph $H$ into $A$ with map $\phi$ is isomorphic to gluing $H$ into $B$ with map $\pi(\phi)$. 

\begin{lem} \label{lem:isomorphism}
Let $G=(A \cup B, E)$ be a bipartite graph with independent sets $A,B$ on $m$ vertices such that there exists a graph automorphism $\pi$ where $\pi(A)=B$, $\pi(B)=A$, and $\pi^2=id$. 

Let $H_A$ and $H_B$ be any two graphs on $m$ vertices, $\phi_A$ be a bijective mapping from $V(H_A)$ to $V(A)$, and $\phi_B$ be a bijective mapping from $V(H_B)$ to $V(B)$. 

Then $G_1=G+\phi_A(E(H_A))+\phi_B(E(H_B))$ is isomorphic to $G_2=G+\pi(\phi_A(E(H_A)))+ \pi(\phi_B(E(H_B)))$.
\end{lem}

\begin{proof}
Let $\pi$ be a graph automorphism of $G=A\cup B$ such that $\pi(A)=B$, $\pi(B)=A$, and $\pi^2=id$. 
Since $\pi$ is a graph automorphism of $G$, then $E(G)= \pi(E(G))$. Consider 

\begin{align*}
    \pi\left[E(G_1)\right]&=
    \pi\left[E(G)+\phi_A(E(H_A))+\phi_B(E(H_B))\right]\\
    &= \pi\left[E(G)\right] + \pi \left[\phi_A(E(H_A))\right]+ \pi\left[\phi_B(E(H_B))\right] \\
    &= E(G) + \pi[\phi_A(E(H_A))]+ \pi[\phi_B(E(H_B))] \\
    &=E(G_2), 
\end{align*}
and
\begin{align*}
    \pi\left[E(G_2)\right]&= \pi\left[E(G) + \pi(\phi_A(E(H_A)))+ \pi(\phi_B(E(H_B)))\right] \\
    &= \pi\left[E(G)\right] + \pi\left[\pi(\phi_A(E(H_A)))\right]+ \pi\left[\pi(\phi_B(E(H_B)))\right] \\
    &= E(G) + \phi_A(E(H_A))+ \phi_B(E(H_B)) \\
    &=E(G_1). 
\end{align*}
Therefore $\pi$ is a graph isomorphism of $G_1,G_2$. 
\end{proof}

In our construction method we will be using this idea of gluing graphs $H_A, H_B$ into $U,W$ in a base graph $G$ two different ways to create a pair of cospectral graphs. If there are initially no edges within $U$ or $W$, then the induced subgraph of G on vertices $U,W$, denoted $G[U \cup W]$, is a bipartite graph. 
As demonstrated in Lemma \ref{lem:isomorphism}, in a bipartite graph $G$, gluing graphs into $G$ in two different ways can create a graph isomorphism between the two new graphs. Therefore, when we glue $H_A,B_B$ in two different ways, the $2m \times 2m$ submatrices representing the vertices of $G[U \cup W]$ are permutation similar. We can always choose our labeling of the two new graphs such that the $2m \times 2m$ submatrix representing the vertices of $G[U \cup W]$ in their corresponding matrices are permutation similar by an anti-diagonal reflection.

\subsection{Linear Algebra Results}
We now discuss some linear algebra tools developed to operate with anti-diagonal reflections. 
When a $n \times n$ matrix $M$ is reflected along its anti-diagonal, we denote this with $\T{M}$ and has $(i,j)$-entry $m_{n-j+1,n-i+1}$. Let $\widehat{I}$ be the $n\times n$ matrix with ones along the anti-diagonal and zeros elsewhere.

\begin{lem} \label{anitDia}
Let $M$ be a matrix in $\mathbb{F}_{n \times n}$. Then $\T{M}=\widehat{I}M^T \widehat{I}$ and $\T{M}$ is similar to $M$.
\end{lem}

\begin{proof}
Consider $x=[x_1, \dots, x_n]$. Therefore 
\begin{align*}
    \widehat{I}x^T&=[x_n, \dots, x_1]^T\\
    x\widehat{I}&=[x_n, \dots, x_1].
\end{align*}
In other words, $\widehat{I}$ reverses the order of a column or row vector when multiplied from the right or left respectively. Therefore $\widehat{I}M$ is the matrix $M$ but where each column is in reverse order (horizontal reflection). And $M\widehat{I}$ is the matrix $M$ but where each row is in reverse order (vertical reflection). 

Therefore, $\widehat{I}M\widehat{I}$ is the matrix $M$ but with all the columns and rows in reverse order. 
 
\begin{align*}
    \widehat{I}M\widehat{I}&=\begin{bmatrix}
m_{n,n} & \cdots & m_{n,1} \\
\vdots & \ddots & \vdots \\
m_{1,n} & \cdots & m_{1,1}
\end{bmatrix}
\end{align*}

To have $M$ be reflected along its anti-diagonal (thus have the anti-diagonal be stationary), we must take the transpose of $\widehat{I}M\widehat{I}$. Therefore $\T{M}=(\widehat{I}M\widehat{I})^T=\widehat{I}M^T\widehat{I}$ since $\widehat{I}$ is its own transpose. 

Note that $(\widehat{I})^2=I$ therefore $\widehat{I}$ is its own inverse. So $M^T$ is similar to $\T{M}$ and we know that $M$ is similar to $M^T$. Thus $M$ is similar to $\T{M}$.
\end{proof}

The next result gives an explicit matrix that only reflects a submatrix along its anti-diagonal given some conditions about the initial matrix. 

\begin{lem} \label{lem:similarity}
Let $M= \begin{bmatrix}
N & Q \\
Q^T & B
\end{bmatrix}$ be a symmetric $m \times m$ matrix with $2k \times 2k$ submatrix $N$ having constant row sums and every column of $Q$ is of the form $[p, \dots, p, r, \dots, r]^T$ where $p,r$ occur $k$ times each. Then, $M$ is similar to $M'=\begin{bmatrix}
\T{N} & Q \\
Q^T & B
\end{bmatrix}$. 
\end{lem}

\begin{proof}
Let $M$ be a matrix with the proprieties stated in the hypothesis. 
Consider the matrix 
\begin{align}\label{matrix}
\mathcal{S}=\begin{bmatrix}
\frac{1}{k}J_{2k}- \widehat{I_{2k}} & 0 \\
0 & I_{m-2k}
\end{bmatrix}
\end{align}
where $J$ is the all ones matrix. Observe that $\mathcal{S}$ is its own inverse and transpose. We will show that $\mathcal{S}$ is a similarity matrix for $M$ and $M'$. 

Therefore, we will show \begin{align*}
    \begin{bmatrix}
\frac{1}{k}J- \widehat{I} & 0 \\
0 & I
\end{bmatrix} \begin{bmatrix}
N & Q \\
Q^T & B
\end{bmatrix}
    \begin{bmatrix}
\frac{1}{k}J- \widehat{I} & 0 \\
0 & I
\end{bmatrix}= \begin{bmatrix}
\T{N} & Q \\
Q^T & B
\end{bmatrix}.
\end{align*}

First, we will show that $\left(\frac{1}{k}J- \widehat{I}\right)N\left(\frac{1}{k}J- \widehat{I}\right)=\T{N}$ and then $\left(\frac{1}{k}J- \widehat{I}\right)Q=Q$. 

We know that $N$ is a symmetric $2k \times 2k$ matrix with constant row and column sums equal to $\alpha$.  Additionally, since $\widehat{I}$ reverses the columns or rows of a matrix, it follows that $J\widehat{I}=J$ and $\widehat{I}J=J$. Using these facts and Lemma \ref{anitDia}, consider the following. 

\begin{align*}
    \left(\frac{1}{k}J- \widehat{I}\right)N\left(\frac{1}{k}J- \widehat{I}\right)
    &=\frac{1}{k^2}JNJ-\frac{1}{k}JN \widehat{I}-\frac{1}{k}\widehat{I}NJ+\widehat{I}N\widehat{I}\\
    &=\frac{2k\alpha}{k^2}J-\frac{\alpha}{k}J-\frac{\alpha}{k}J+ \widehat{I}N\widehat{I}\\
    &=\widehat{I}N\widehat{I}\\
    &=\widehat{I}N^T\widehat{I}\\
    &=\T{N}.
\end{align*}

The columns of $Q$ all have the form $[p, \dots, p, r, \dots, r]^T$ where $p,r$ occur $k$ times each. Therefore

\begin{align*}
    (\frac{1}{k}J-\widehat{I})[p, \dots, p, r, \dots, r]^T &= \frac{1}{k}J [p, \dots, p, r, \dots, r]^T- \widehat{I}[p, \dots, p, r, \dots, r]^T \\
    &= \frac{1}{k}[k(p+r), \dots, k(p+r)]^T - [r, \dots, r, p, \dots, p]^T\\
    &= [p, \dots, p, r, \dots, r]^T.
\end{align*}
\end{proof}

We now are ready to provide graph structural conditions that create cospectral graphs that have $\mathcal{S}$ (as defined in (1)) as a similarity matrix. 

\subsection{Construction method}

The method of creating cospectral graphs for the distance Laplacian given in \cite{grwc} not only generalizes to larger families of graphs, it also extends to other matrices. This is because the proof used the fact that the distance Laplacian has constant row sums equal to zero. Thus it is very natural to extend it to the combinatorial Laplacian. 

The constructions outline necessary graph conditions such that Lemma \ref{lem:similarity} can be applied to the matrix. Therefore, our argument will be that the columns of $Q$ are of the appropriate form and that gluing our graphs $H_1, H_2$ into $G$ in two ways results in the anti-diagonal flip of $N$. 

Let $G[U]$ be the induced subgraph of $G$ on the vertices of $U \subseteq V(G)$. 

\begin{thm} \label{cousinConst}
Let $G$ be graph containing two vertex sets $V_1, V_2$ each on $m$ vertices such that 

\begin{enumerate}
    \item $G[V_1],G[V_2]$ are empty subgraphs;
    \item there exists a graph automorphism $\pi$ for $G[V_1 \cup V_2]$ such that $\pi(V_1)=V_2$, $\pi(V_2)=V_1$, and $\pi^2=id$. 
\end{enumerate}
Let $H_1$ and $H_2$ be any two graphs on $m$ vertices and $\phi_i$ be a bijective mapping from $H_i$ to $V_i$ for $i \in \{1,2\}$. Let $G_1=G+ \phi_{1}(E(H_1)) + \phi_{2}(E(H_2)))$  and $G_2=G+ \pi(\phi_{1}(E(H_1))) + \pi(\phi_{2}(E(H_2))))$. 

\begin{itemize}
    \item If $V_1$ and $V_2$ are co-transmission cousins, then $G_1$ and $G_2$ are cospectral for the distance Laplacian matrix.
    
    \item If $V_1$ and $V_2$ are cousins and $G_1[V_1 \cup V_2]$ is a transmission regular graph, then $G_1$ and $G_2$ are cospectral for the distance matrix.
    
    \item If $V_1$ and $V_2$ are co-degree cousins, then $G_1$ and $G_2$ are cospectral for the combinatorial Laplacian matrix. 
    
    \item If $V_1$ and $V_2$ are co-degree cousins and $G_1[V_1 \cup V_2]$ is a regular graph, then $G_1$ and $G_2$ are cospectral for the signless Laplacian matrix.
    
    \item If $V_1$ and $V_2$ are relaxed cousins and $G_1[V_1 \cup V_2]$ is a regular graph, then $G_1$ and $G_2$ are cospectral for the adjacency matrix. 
\end{itemize}  
\end{thm}

\begin{proof}
First consider the cases for the distance matrices where $V_1,V_2$ are cousins (since co-transmission cousins are a special case of cousins).  
By Lemma \ref{notInPath}, we know that no shortest path uses an edge with both of its endpoints in $V_1$ or $V_2$. Thus no distance between pairs of vertices where at least one is not in $V_1$ nor $V_2$ changes. 

In the case of relaxed or co-degree cousins, we note that no adjacency changes between a pair of vertices where at least one is not in $V_1$ nor $V_2$ when we add edges with both of its endpoints in $V_1$ or $V_2$. 

So we can partition the respective matrix $M$ of the two graphs $G_1,G_2$ into \begin{align*} \begin{bmatrix}
M_1 & Q \\
Q^T & B
\end{bmatrix} and \begin{bmatrix}
M_2 & Q \\
Q^T & B
\end{bmatrix}\end{align*} 
where $M_1,M_2$ are $2m \times 2m$ submatrices that are indexed by the vertices in $V_1$ followed by the vertices in $V_2$. 

We claim that with the appropriate labeling $M_1$ and $M_2$ are permutation similar submatrices where the permutation is an anti-diagonal reflection. In other words, $\T{M_1}=M_2$. 

Let $\pi$ be a graph automorphism of $G[V_1 \cup V_2]$ such that $\pi(V_1)=V_2$, $\pi(V_2)=V_1$, and $\pi^2=id$. 
We can relabel the vertices of $V_2$ in $G$ such that $\pi(v_{1,i})=v_{2,m-i+1}$ and $\pi(v_{2,i})=v_{1,m-i+1}$ for $v_{1,i} \in V_1$ and $v_{2,i} \in V_2$ since $\pi$ is an involution.
By Lemma \ref{lem:isomorphism} $\pi$ is a graph isomorphism of $G_1[V_1\cup V_2],G_2[V_1 \cup V_2]$ and $M_1,M_2$ are permutation similar matrices. 
Since we chose $\pi$ such that $\pi(v_{1,i})=v_{2,m-i+1}$ and $\pi(v_{2,i})=v_{1,m-i+1}$, it follows that $\T{M_1}$ and $M_2$ are equivalent.

For each case, we claim that $M_1$ (and $M_2$) is a symmetric $2m \times 2m$ matrix with constant row and column sums and the columns of $Q$ all have the form $[p, \dots, p, r, \dots, r]^T$ where $p, r$ appear $m$-times. 

\begin{itemize}
    \item For the distance Laplacian matrix, we know $V_1,V_2$ are co-transmission cousins which means that $Q$ has constant row sums and has columns of the form $[p, \dots, p, r, \dots, r]^T$ where $p, r$ appear $m$-times. And the distance Laplacian has row sums equal to zero, therefore $M_1$ has constant row sums. 
    
    \item For the distance matrix, we know $V_1,V_2$ are cousins which means that $Q$ has columns of the form $[p, \dots, p, r, \dots, r]^T$ where $p, r$ appear $m$-times. Since $G_1[V_1 \cup V_2]$ is transmission regular, it follows that $M_1$ has constant row sums.
    
    \item For the combinatorial Laplacian matrix, we know $V_1,V_2$ are co-degree cousins which means that $Q$ has constant row sums and has columns of the form $[p, \dots, p, r, \dots, r]^T$ where $p, r\in \{0,-1\}$ appear $m$-times.  And the combinatorial Laplacian has row sums equal to zero, therefore $M_1$ has constant row sums. 
    
    \item For the signless Laplacian matrix, we know $V_1, V_2$ are co-degree cousins which means that $Q$ has constant row sums and has columns of the form $[p, \dots, p, r, \dots, r]^T$ where $p, r\in \{0,1\}$ appear $m$-times. Since $G_1[V_1 \cup V_2]$ is a regular graph, the sum of the rows using only the non-diagonal entries of $M_1$ is constant. We know that the diagonal entries of $M_1$ are the sums of the non diagonal entries of $[M_1, Q]$. This is a constant because the sum non-diagonal entries of $M_1$ is constant and $Q$ has constant row sums. Therefore $M_1$ has constant diagonal entries and moreover constant row sums. 
    
    
    
    \item For the adjacency matrix, we know $V_1, V_2$ are relaxed cousins which means that $Q$ has columns of the form $[p, \dots, p, r, \dots, r]^T$ where $p, r\in \{0,1\}$ appear $m$-times. Since $G_1[V_1 \cup V_2]$ is a regular graph, it follows that $M_1$ has constant row sums. 
\end{itemize}

Thus by Lemma \ref{lem:similarity}, $\mathcal{S}$ is a similarity matrix for $M$ of $G_1,G_2$. Therefore $G_1,G_2$ are cospectral for $M$. 
\end{proof}

This allows us to create many different cospectral graphs on several matrices. Next we present some examples of this construction.

\begin{ex}
Consider the graphs in Figure \ref{fig:DL2}. The graph in (a) is a graph $G$  with co-transmission cousins $V_1=\{v_{1,1}, v_{1,2}, v_{1,3}, v_{1,4}\}$ and $V_2=\{v_{2,1}, v_{2,2}, v_{2,3}, v_{2,4}\}$. Let $\pi$ be a map from $V_1 \cup V_2 \to V_1 \cup V_2$ where $\pi(v_{1,j})=v_{2,5-j}$ and $\pi(v_{2,j})=v_{1,5-j}$. For our graph $G$ we can see that this is a graph automorphism for $G[V_1 \cup V_2]$. 

Therefore, by Theorem \ref{cousinConst} we can glue any two graphs on four vertices into $V_1, V_2$ and then into $V_2, V_1$ with respect to $\pi$ to create a pair of cospectral graphs for the distance Laplacian.  

In (b) and (c) we have glued the paw graph and $K_2$ into the cousin sets in two different ways to create cospectral graphs. 
\end{ex}
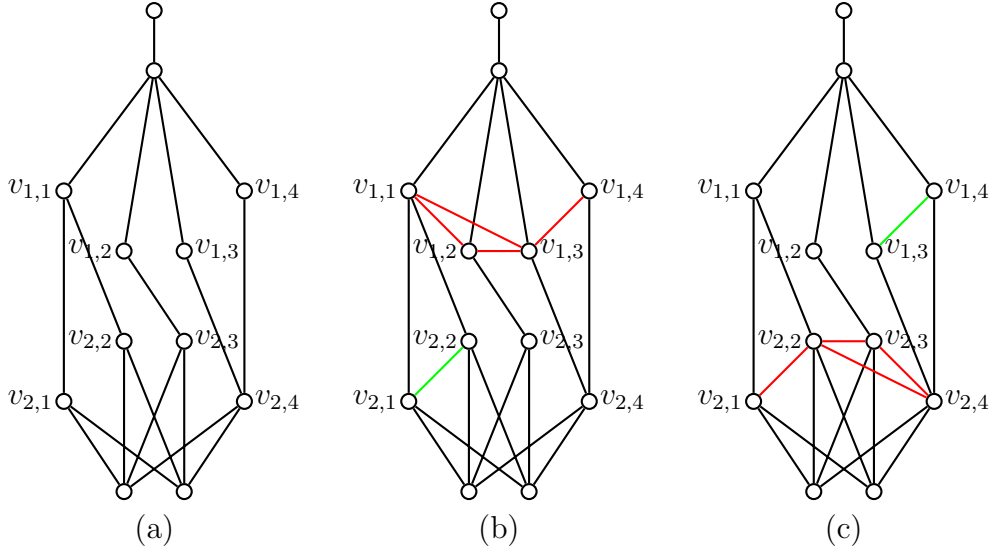
\begin{figure}[h]
    \centering
    \begin{tabular}{ccc}
        \begin{tikzpicture}[scale=.8]
    \node[vertex] (w1) at (0,8) {};
    \node[vertex] (w2) at (0,7) {};
    
    \node[vertex] (v4) at (1.5,5) {};
    \node[vertex] (v1) at (-1.5, 5) {};
    \node[vertex] (v3) at (0.5, 4) {};
    \node[vertex] (v2) at (-0.5, 4) {};
    
    \node[vertex] (u3) at (0.5, 2.5) {};
    \node[vertex] (u1) at (-1.5, 1.5) {};
    \node[vertex] (u4) at (1.5, 1.5) {};
    \node[vertex] (u2) at (-0.5, 2.5) {};
    
    \node[vertex] (w3) at (.5, 0) {};
    \node[vertex] (w4) at (-.5, 0) {};
    
    \node[left] at (v1) {$v_{1,1}$};
    \node[left] at (v2) {$v_{1,2}$};
    \node[right] at (v3) {$v_{1,3}$};
    \node[right] at (v4) {$v_{1,4}$};
    
    \node[right] at (u4) {$v_{2,4}$};
    \node[right] at (u3) {$v_{2,3}$};
    \node[left] at (u2) {$v_{2,2}$};
    \node[left] at (u1) {$v_{2,1}$};
    
    \draw[thick] (w1)--(w2)--(v1)
    (v3)--(w2)--(v2)
    (w2)--(v4)
    (w3)--(u1)--(w4)
    (w3)--(u2)--(w4)
    (w3)--(u3)--(w4)
    (w3)--(u4)--(w4);
    
    \draw[thick] (u2)--(v1)--(u1)
    (u3)--(v2)
    (v3)--(u4)--(v4);
    
    \end{tikzpicture} &  \begin{tikzpicture}[scale=.8]
    \node[vertex] (w1) at (0,8) {};
    \node[vertex] (w2) at (0,7) {};
    
    \node[vertex] (v4) at (1.5,5) {};
    \node[vertex] (v1) at (-1.5, 5) {};
    \node[vertex] (v3) at (0.5, 4) {};
    \node[vertex] (v2) at (-0.5, 4) {};
    
    \node[vertex] (u3) at (0.5, 2.5) {};
    \node[vertex] (u1) at (-1.5, 1.5) {};
    \node[vertex] (u4) at (1.5, 1.5) {};
    \node[vertex] (u2) at (-0.5, 2.5) {};
    
    \node[vertex] (w3) at (.5, 0) {};
    \node[vertex] (w4) at (-.5, 0) {};
    
    \node[left] at (v1) {$v_{1,1}$};
    \node[left] at (v2) {$v_{1,2}$};
    \node[right] at (v3) {$v_{1,3}$};
    \node[right] at (v4) {$v_{1,4}$};
    
    \node[right] at (u4) {$v_{2,4}$};
    \node[right] at (u3) {$v_{2,3}$};
    \node[left] at (u2) {$v_{2,2}$};
    \node[left] at (u1) {$v_{2,1}$};
    
    \draw[thick] (w1)--(w2)--(v1)
    (v3)--(w2)--(v2)
    (w2)--(v4)
    (w3)--(u1)--(w4)
    (w3)--(u2)--(w4)
    (w3)--(u3)--(w4)
    (w3)--(u4)--(w4);
    
    \draw[thick] (u2)--(v1)--(u1)
    (u3)--(v2)
    (v3)--(u4)--(v4);
    
    \draw[thick, red] (v3)--(v2)--(v1)--(v3)--(v4);
    \draw[thick, green] (u1)--(u2);
    \end{tikzpicture} & \begin{tikzpicture}[scale=.8]
   \node[vertex] (w1) at (0,8) {};
    \node[vertex] (w2) at (0,7) {};
    
    \node[vertex] (v4) at (1.5,5) {};
    \node[vertex] (v1) at (-1.5, 5) {};
    \node[vertex] (v3) at (0.5, 4) {};
    \node[vertex] (v2) at (-0.5, 4) {};
    
    \node[vertex] (u3) at (0.5, 2.5) {};
    \node[vertex] (u1) at (-1.5, 1.5) {};
    \node[vertex] (u4) at (1.5, 1.5) {};
    \node[vertex] (u2) at (-0.5, 2.5) {};
    
    \node[vertex] (w3) at (.5, 0) {};
    \node[vertex] (w4) at (-.5, 0) {};
    
    \node[left] at (v1) {$v_{1,1}$};
    \node[left] at (v2) {$v_{1,2}$};
    \node[right] at (v3) {$v_{1,3}$};
    \node[right] at (v4) {$v_{1,4}$};
    
    \node[right] at (u4) {$v_{2,4}$};
    \node[right] at (u3) {$v_{2,3}$};
    \node[left] at (u2) {$v_{2,2}$};
    \node[left] at (u1) {$v_{2,1}$};
    
    \draw[thick] (w1)--(w2)--(v1)
    (v3)--(w2)--(v2)
    (w2)--(v4)
    (w3)--(u1)--(w4)
    (w3)--(u2)--(w4)
    (w3)--(u3)--(w4)
    (w3)--(u4)--(w4);
    
    \draw[thick] (u2)--(v1)--(u1)
    (u3)--(v2)
    (v3)--(u4)--(v4);
    
    \draw[thick, red] (u2)--(u3)--(u4)--(u2)--(u1);
    \draw[thick, green] (v4)--(v3);
    \end{tikzpicture}\\
      (a)   &  (b) & (c)
    \end{tabular}
    \caption{(a) A graph $G$ with $V_1=\{v_{1,1}, v_{1,2}, v_{1,3}, v_{1,4}\}$ a set of co-transmission cousins to $V_2=\{v_{2,1}, v_{2,2}, v_{2,3}, v_{2,4}\}$. (b) The graph constructed by $G$ with $H_1$ (paw graph) glued into $V_1$ and $H_2=K_2+\{v_3,v_4\}$ glued into $V_2$. (c) The graph constructed by $G$ with $H_2$ glued into $V_1$ and $H_1$ glued into $V_2$. By Theorem \ref{cousinConst} the graphs show in (b) and (c) are cospectral for the distance Laplacian.}
    \label{fig:DL2}
\end{figure}

\begin{ex}
Consider the graphs in Figure \ref{fig:transmission regular}. The graph in (a) is a graph $G$ with co-degree cousins $V_1=\{v_{1,1}, v_{1,2}, v_{1,3}, v_{1,4}, v_{1,5}, v_{1,6}\}$ and $V_2=\{v_{2,1}, v_{2,2}, v_{2,3}, v_{2,4}, v_{2,5}, v_{1,6}\}$. Let $\pi$ be a map from $V_1 \cup V_2 \to V_1 \cup V_2$ where $\pi(v_{1,j})=v_{2,7-j}$ and $\pi(v_{2,j})=v_{1,7-j}$. For our graph $G$ we can see that this a graph automorphism for $G[V_1 \cup V_2]$. 

Therefore, by Theorem \ref{cousinConst} we can glue any two graphs on six vertices into $V_1,V_2$ and then into $V_2, V_1$ with respect to $\pi$ to create a pair of cospectral graphs for the combinatorial Laplacian. 

Since we glued in two non-isomorphic $3$-regular graphs such that $G_1[V_1 \cup V_2]$ in (b) is a regular graph, we also create a pair of cospectral graphs for the signless Laplacian and adjacency matrix. Additionally, since $G_1[V_1 \cup V_2]$ has diameter $2$ and is regular, it is also transmission regular. Therefore, this pair of graphs in (b) and (c) are also cospectral for the distance matrix. 
\end{ex}

We can have pairs of cospectral graphs using Theorem \ref{cousinConst} for the adjacency matrix without being cospectral for the signless Laplacian. Figure \ref{fig:CousWithRegularity} gives such a pair.

Since the conditions for the signless Laplacian construction are a special case of the conditions for the adjacency matrix and a special case of the combinatorial Laplacian construction, it follows that anytime we have a pair of graphs that are cospectral for the signless Laplacian using Theorem \ref{cousinConst} these graphs are also cospectral for the adjacency matrix and the combinatorial Laplacian.

\begin{figure}[h]
    \centering
    \begin{tabular}{ccc}
        \begin{tikzpicture}[scale=.8]
    \node[vertex] (w1) at (0,8) {};
    \node[vertex] (w2) at (0,7) {};
    
    \node[vertex] (v1) at (2, 5) {};
    \node[vertex] (v2) at (0, 5) {};
    \node[vertex] (v3) at (-2, 5) {};
    \node[vertex] (v4) at (-2,4) {};
    \node[vertex] (v5) at (0, 4) {};
    \node[vertex] (v6) at (2,4) {};
    
    \node[vertex] (u1) at (2, 1) {};
    \node[vertex] (u2) at (0, 1) {};
    \node[vertex] (u3) at (-2, 1) {};
    \node[vertex] (u4) at (-2,0) {};
    \node[vertex] (u5) at (0, 0) {};
    \node[vertex] (u6) at (2,0) {};
    
     \node[vertex] (w3) at (0, -2) {};
    
    \node[above] at (v1) {$v_{1,1}$};
    \node[above] at (v2) {$v_{1,2}$};
    \node[above] at (v3) {$v_{1,3}$};
    \node[below] at (v4) {$v_{1,4}$};
    \node[below] at (v5) {$v_{1,5}$};
    \node[below] at (v6) {$v_{1,6}$};
    
    \node[below] at (u6) {$v_{2,1}$};
    \node[below] at (u5) {$v_{2,2}$};
    \node[below] at (u4) {$v_{2,3}$};
    \node[above] at (u3) {$v_{2,4}$};
    \node[above] at (u2) {$v_{2,5}$};
    \node[above] at (u1) {$v_{2,6}$};
    
    \draw[thick] (0,4.5) ellipse (3cm and 1.5cm);
    \draw[thick] (0,0.5) ellipse (3cm and 1.5cm);
    
    \draw[thick] (w1)--(w2);
    
    \draw[line width=4pt]
    (0,6)--(w2)
    (0,-1)--(w3)
    (0, 3) -- (0,2);
    
    \end{tikzpicture} & \begin{tikzpicture}[scale=.8]
    \node[vertex] (w1) at (0,8) {};
    \node[vertex] (w2) at (0,7) {};
    
    \node[vertex] (v1) at (2, 5) {};
    \node[vertex] (v2) at (0, 5) {};
    \node[vertex] (v3) at (-2, 5) {};
    \node[vertex] (v4) at (-2,4) {};
    \node[vertex] (v5) at (0, 4) {};
    \node[vertex] (v6) at (2,4) {};
    
    \node[vertex] (u1) at (2, 1) {};
    \node[vertex] (u2) at (0, 1) {};
    \node[vertex] (u3) at (-2, 1) {};
    \node[vertex] (u4) at (-2,0) {};
    \node[vertex] (u5) at (0, 0) {};
    \node[vertex] (u6) at (2,0) {};
    
     \node[vertex] (w3) at (0, -2) {};
    
    \node[above] at (v1) {$v_{1,1}$};
    \node[above] at (v2) {$v_{1,2}$};
    \node[above] at (v3) {$v_{1,3}$};
    \node[below] at (v4) {$v_{1,4}$};
    \node[below] at (v5) {$v_{1,5}$};
    \node[below] at (v6) {$v_{1,6}$};
    
    \node[below] at (u6) {$v_{2,1}$};
    \node[below] at (u5) {$v_{2,2}$};
    \node[below] at (u4) {$v_{2,3}$};
    \node[above] at (u3) {$v_{2,4}$};
    \node[above] at (u2) {$v_{2,5}$};
    \node[above] at (u1) {$v_{2,6}$};
    
    \draw[thick] (0,4.5) ellipse (3cm and 1.5cm);
    \draw[thick] (0,0.5) ellipse (3cm and 1.5cm);
    
    \draw[thick] (w1)--(w2);
    
    \draw[line width=4pt]
    (0,6)--(w2)
    (0,-1)--(w3)
    (0, 3) -- (0,2);
    
    \draw[thick, red] (v1)--(v2)--(v3)--(v4)--(v5)--(v6)--(v1)--(v4)
    (v3)--(v6)
    (v2)--(v5);
    \draw[thick, green] (u1)--(u2)--(u6)--(u3)--(u2)
    (u6)--(u5)--(u1)--(u4)--(u5)
    (u3)--(u4);
    \end{tikzpicture}& \begin{tikzpicture}[scale=.8]
    \node[vertex] (w1) at (0,8) {};
    \node[vertex] (w2) at (0,7) {};
    
    \node[vertex] (v1) at (2, 5) {};
    \node[vertex] (v2) at (0, 5) {};
    \node[vertex] (v3) at (-2, 5) {};
    \node[vertex] (v4) at (-2,4) {};
    \node[vertex] (v5) at (0, 4) {};
    \node[vertex] (v6) at (2,4) {};
    
    \node[vertex] (u1) at (2, 1) {};
    \node[vertex] (u2) at (0, 1) {};
    \node[vertex] (u3) at (-2, 1) {};
    \node[vertex] (u4) at (-2,0) {};
    \node[vertex] (u5) at (0, 0) {};
    \node[vertex] (u6) at (2,0) {};
    
     \node[vertex] (w3) at (0, -2) {};
    
    \node[above] at (v1) {$v_{1,1}$};
    \node[above] at (v2) {$v_{1,2}$};
    \node[above] at (v3) {$v_{1,3}$};
    \node[below] at (v4) {$v_{1,4}$};
    \node[below] at (v5) {$v_{1,5}$};
    \node[below] at (v6) {$v_{1,6}$};
    
    \node[below] at (u6) {$v_{2,1}$};
    \node[below] at (u5) {$v_{2,2}$};
    \node[below] at (u4) {$v_{2,3}$};
    \node[above] at (u3) {$v_{2,4}$};
    \node[above] at (u2) {$v_{2,5}$};
    \node[above] at (u1) {$v_{2,6}$};
    
    \draw[thick] (0,4.5) ellipse (3cm and 1.5cm);
    \draw[thick] (0,0.5) ellipse (3cm and 1.5cm);
    
    \draw[thick] (w1)--(w2);
    
    \draw[line width=4pt]
    (0,6)--(w2)
    (0,-1)--(w3)
    (0, 3) -- (0,2);
    
    \draw[thick, red] (u1)--(u2)--(u3)--(u4)--(u5)--(u6)--(u1)--(u4)
    (u3)--(u6)
    (u2)--(u5);
    \draw[thick, green] (v1)--(v2)--(v6)--(v3)--(v2)
    (v6)--(v5)--(v1)--(v4)--(v5)
    (v3)--(v4);
    \end{tikzpicture} \\
        (a) & (b) & (c)
    \end{tabular}
    
    \caption{Let the thick edges represent all possible edges between a vertex and a cluster of vertices. (a) A graph $G$ with $V_1=\{v_{1,1}, v_{1,2}, v_{1,3}, v_{1,4}, v_{1,5}, v_{1,6}\}$ and $V_2=\{v_{2,1}, v_{2,2}, v_{2,3}, v_{2,4}, v_{2,5}, v_{1,6}\}$ which are (co-degree) cousins. (b) A graph $G_1$ constructed from $G$ such that $G_1[V_1 \cup V_2]$ is a regular graph. Therefore, it is has the same generalized characteristic polynomial as $G_2$ shown in (c) by Corollary \ref{cor:genCharPoly}. Since $G_1[V_1 \cup V_2]$ is a regular graph with $diam=2$, it is transmission regular. Therefore this pair is cospectral for the distance matrix by Theorem \ref{cousinConst}.  }
    \label{fig:transmission regular}
\end{figure}
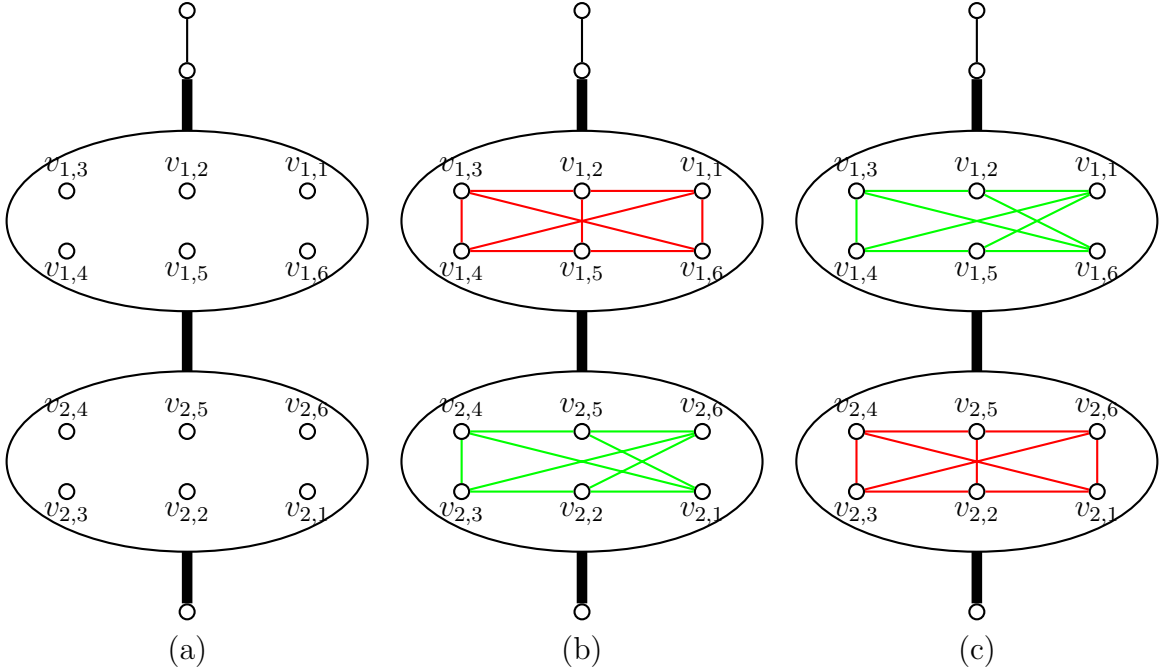

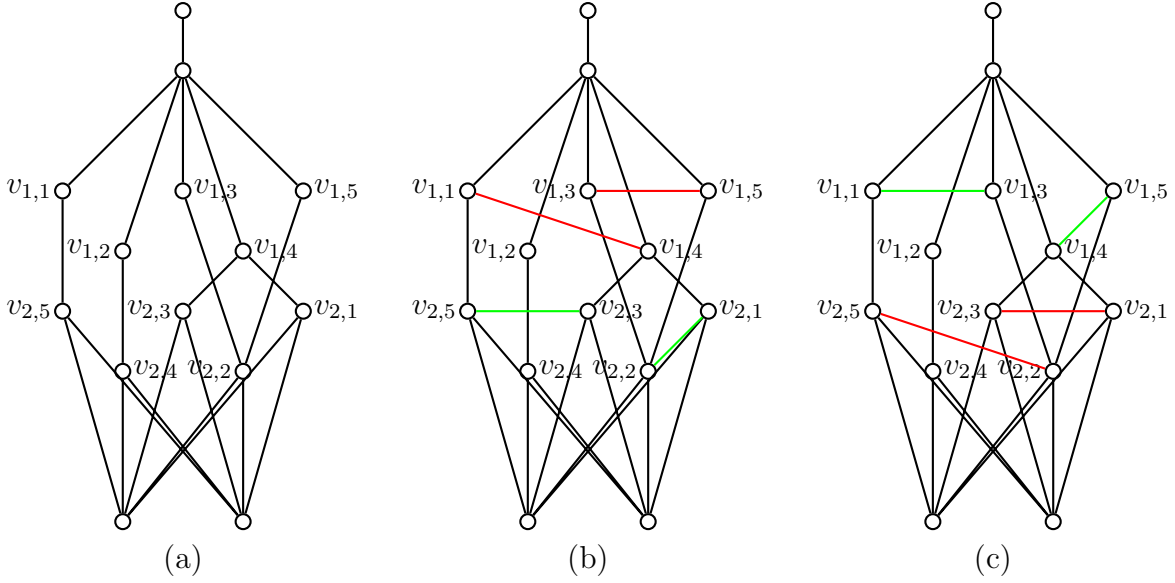
\begin{figure}[h]
    \centering
    \begin{tabular}{ccc}
         \begin{tikzpicture}[scale=.8]
    \node[vertex] (w1) at (0,8) {};
    \node[vertex] (w2) at (0,7) {};
    
    \node[vertex] (v4) at (1,4) {};
    \node[vertex] (v2) at (-1, 4) {};
    \node[vertex] (v5) at (2, 5) {};
    \node[vertex] (v3) at (0, 5) {};
    \node[vertex] (v1) at (-2, 5) {};
    
    \node[vertex] (u4) at (1, 2) {};
    \node[vertex] (u2) at (-1, 2) {};
    \node[vertex] (u5) at (2, 3) {};
    \node[vertex] (u3) at (0, 3) {};
    \node[vertex] (u1) at (-2, 3) {};
    
     \node[vertex] (w3) at (-1, -0.5) {};
    \node[vertex] (w4) at (1, -0.5) {};
    
    \node[left] at (v1) {$v_{1,1}$};
    \node[left] at (v2) {$v_{1,2}$};
    \node[right] at (v3) {$v_{1,3}$};
    \node[right] at (v4) {$v_{1,4}$};
    \node[right] at (v5) {$v_{1,5}$};
    
    \node[right] at (u5) {$v_{2,1}$};
    \node[left] at (u4) {$v_{2,2}$};
    \node[left] at (u3) {$v_{2,3}$};
    \node[right] at (u2) {$v_{2,4}$};
    \node[left] at (u1) {$v_{2,5}$};
    
    \draw[thick] (w1)--(w2)--(v1)
    (v3)--(w2)--(v2)
    (v5)--(w2)--(v4)
    (w3)--(u1)--(w4)
    (w3)--(u2)--(w4)
    (w3)--(u3)--(w4)
    (w3)--(u4)--(w4)
    (w3)--(u5)--(w4);
    
    \draw[thick] (v5)--(u4)--(v3)
    (u5)--(v4)--(u3)
    (v1)--(u1)
    (v2)--(u2);
    
    \end{tikzpicture} &  \begin{tikzpicture}[scale=.8]
    \node[vertex] (w1) at (0,8) {};
    \node[vertex] (w2) at (0,7) {};
    
    \node[vertex] (v4) at (1,4) {};
    \node[vertex] (v2) at (-1, 4) {};
    \node[vertex] (v5) at (2, 5) {};
    \node[vertex] (v3) at (0, 5) {};
    \node[vertex] (v1) at (-2, 5) {};
    
    \node[vertex] (u4) at (1, 2) {};
    \node[vertex] (u2) at (-1, 2) {};
    \node[vertex] (u5) at (2, 3) {};
    \node[vertex] (u3) at (0, 3) {};
    \node[vertex] (u1) at (-2, 3) {};
    
     \node[vertex] (w3) at (-1, -0.5) {};
    \node[vertex] (w4) at (1, -0.5) {};
    
    \node[left] at (v1) {$v_{1,1}$};
    \node[left] at (v2) {$v_{1,2}$};
    \node[left] at (v3) {$v_{1,3}$};
    \node[right] at (v4) {$v_{1,4}$};
    \node[right] at (v5) {$v_{1,5}$};
    
    \node[right] at (u5) {$v_{2,1}$};
    \node[left] at (u4) {$v_{2,2}$};
    \node[right] at (u3) {$v_{2,3}$};
    \node[right] at (u2) {$v_{2,4}$};
    \node[left] at (u1) {$v_{2,5}$};
    
    \draw[thick] (w1)--(w2)--(v1)
    (v3)--(w2)--(v2)
    (v5)--(w2)--(v4)
    (w3)--(u1)--(w4)
    (w3)--(u2)--(w4)
    (w3)--(u3)--(w4)
    (w3)--(u4)--(w4)
    (w3)--(u5)--(w4);
    
    \draw[thick] (v5)--(u4)--(v3)
    (u5)--(v4)--(u3)
    (v1)--(u1)
    (v2)--(u2);
    
    \draw[thick, red] (v1)--(v4) (v3)--(v5);
    \draw[thick, green] (u1)--(u3) (u4)--(u5);
    \end{tikzpicture} & \begin{tikzpicture}[scale=.8]
    \node[vertex] (w1) at (0,8) {};
    \node[vertex] (w2) at (0,7) {};
    
    \node[vertex] (v4) at (1,4) {};
    \node[vertex] (v2) at (-1, 4) {};
    \node[vertex] (v5) at (2, 5) {};
    \node[vertex] (v3) at (0, 5) {};
    \node[vertex] (v1) at (-2, 5) {};
    
    \node[vertex] (u4) at (1, 2) {};
    \node[vertex] (u2) at (-1, 2) {};
    \node[vertex] (u5) at (2, 3) {};
    \node[vertex] (u3) at (0, 3) {};
    \node[vertex] (u1) at (-2, 3) {};
    
     \node[vertex] (w3) at (-1, -0.5) {};
    \node[vertex] (w4) at (1, -0.5) {};
    
    \node[left] at (v1) {$v_{1,1}$};
    \node[left] at (v2) {$v_{1,2}$};
    \node[right] at (v3) {$v_{1,3}$};
    \node[right] at (v4) {$v_{1,4}$};
    \node[right] at (v5) {$v_{1,5}$};
    
    \node[right] at (u5) {$v_{2,1}$};
    \node[left] at (u4) {$v_{2,2}$};
    \node[left] at (u3) {$v_{2,3}$};
    \node[right] at (u2) {$v_{2,4}$};
    \node[left] at (u1) {$v_{2,5}$};
    
    \draw[thick] (w1)--(w2)--(v1)
    (v3)--(w2)--(v2)
    (v5)--(w2)--(v4)
    (w3)--(u1)--(w4)
    (w3)--(u2)--(w4)
    (w3)--(u3)--(w4)
    (w3)--(u4)--(w4)
    (w3)--(u5)--(w4);
    
    \draw[thick] (v5)--(u4)--(v3)
    (u5)--(v4)--(u3)
    (v1)--(u1)
    (v2)--(u2);
    
    \draw[thick, red] (u1)--(u4) (u3)--(u5);
    \draw[thick, green] (v1)--(v3) (v4)--(v5);
    \end{tikzpicture}\\
     (a)    & (b) & (c)
    \end{tabular}
    \caption{(a) A graph $G$ with $V_1=\{v_{1,1}, v_{1,2}, v_{1,3}, v_{1,4}, v_{1,5}\}$ a set of relaxed cousins to $V_2=\{v_{2,1}, v_{2,2}, v_{2,3}, v_{2,4}, v_{2,5}\}$. In this case, they are also co-transmission cousins. (b) A graph $G_1$ constructed from $G$ with $H_1=H_2=K_2 \cup K_2$ glued into $V_1$ and $V_2$. (c) A graph $G_2$ constructed from $G$ with $H_1=H_2=K_2 \cup K_2$ glued into $V_1$ and $V_2$. In (b) $G_1[V_1 \cup V_2]$ is a regular graph, therefore by Theorem \ref{cousinConst} (b) and (c) are cospectral for the adjacency matrix and distance Laplacian matrix. }
    \label{fig:CousWithRegularity}
\end{figure}

The adjacency matrix, combinatorial Laplacian, and signless Laplacian can be related using the \emph{generalized characteristic polynomial.} The generalized characteristic polynomial of a graph $G$, denoted $\phi_G(\lambda,r)$ is the determinant of the matrix \begin{align}
   N_G(\lambda,r) = \lambda I_n-A_G+rD_G.
\end{align}
It can be beneficial to use this matrix because we can write our adjacency, combinatorial Laplacian, signless Laplacian, and normalized Laplacian in terms of this matrix. For example if $p_M(\lambda)$ is the characteristic polynomial of $M$, then \begin{align*}
    p_A(\lambda)&=\phi_G(\lambda,0) \\
    p_L(\lambda)&=\phi_G(-\lambda, 1)\\
    p_{|L|}(\lambda)&=\phi_G(\lambda, -1)\\
    p_{\mathcal{L}}(\lambda)&=\frac{(-1)^{|V|}}{\det(D_G)} \phi_G(0,-\lambda+1).
\end{align*} 

Therefore if $\phi_G(\lambda,r)=\phi_H(\lambda,r)$, then graphs $G,H$ are cospectral for the adjacency, Laplacian, signless Laplacian, and normalized Laplacian. We have seen our construction classify pairs of graphs that are cospectral for three of these matrices and now we extend it to the generalized characteristic polynomial. 

\begin{cor} \label{cor:genCharPoly}
Let $G$ be a graph containing two vertex sets $V_1, V_2$ each on $m$ vertices such that 

\begin{enumerate}
    \item $G[V_1],G[V_2]$ are empty subgraphs;
    \item there exists a graph automorphism $\pi$ for $G[V_1 \cup V_2]$ such that $\pi(V_1)=V_2$, $\pi(V_2)=V_1$, and $\pi^2=id$. 
\end{enumerate}
Let $H_1$ and $H_2$ be any two graphs on $m$ vertices and $\phi_i$ be a bijective mapping from $H_i$ to $V_i$ for $i \in \{1,2\}$. Let $G_1=G+ \phi_{1}(E(H_1)) + \phi_{2}(E(H_2)))$  and $G_2=G+ \pi(\phi_{1}(E(H_1))) + \pi(\phi_{2}(E(H_2))))$. 

If $V_1$ is a set of co-degree cousins to $V_2$ and $G_1[V_1 \cup V_2]$ is a regular graph, then $\phi_{G_1}(\lambda,r)=\phi_{G_2}(\lambda,r)$ where $\phi(\lambda, r)$ is the generalized characteristic polynomial. 
\end{cor}

\begin{proof}
Since our off-diagonal entries of $N_{G}(\lambda,r)$ are the off-diagonal entries of $A_G$, it follows that no adjacency changes between a pair of vertices where at least one is not in $V_1$ nor $V_2$ when we add edges with both of its endpoints in $V_1$ or $V_2$.  
Therefore we can partition the respective matrix $N_{G_1}(\lambda,r),N_{G_2}(\lambda,r)$ of the two graphs into \begin{align*} \begin{bmatrix}
M_1 & Q \\
Q^T & B
\end{bmatrix} and \begin{bmatrix}
M_2 & Q \\
Q^T & B
\end{bmatrix}\end{align*} 
where $M_1,M_2$ are $2m \times 2m$ submatrices that are indexed by the vertices in $V_1$ followed by the vertices in $V_2$. 

In an analogous argument in the proof of Theorem \ref{cousinConst}, we know that that we can always label the graphs such that $\T{M_1}=M_2$. 

We claim that $M_1$ (and $M_2$) is a symmetric $2m \times 2m$ matrix with constant row and column sums and the columns of $Q$ all have the form $[p, \dots, p, r, \dots, r]^T$ where $p, r$ appear $m$-times. 

We know that $M_1$ is $\lambda I_n- A+rD$ restricted to the vertices $V_1,V_2$ and $Q$ is similarly defined. Since $I_n, D$ are both diagonal matrices, the entries of $Q$ are only from the adjacency matrix of $G_1$. 

We know $V_1, V_2$ are co-degree cousins which means that $Q$ has constant row sums and has columns of the form $[p, \dots, p, r, \dots, r]^T$ where $p, r\in \{0,-1\}$ appear $m$-times. 

Since $G_1[V_1 \cup V_2]$ is a regular graph, the sum of the rows using only the non-diagonal entries of $M_1$ is constant since these entries are only from the adjacency matrix. We know that the diagonal entries of $M_1$ are $\lambda$ plus $r$ times the sums of the non-diagonal entries of $[M_1, Q]$. This is a constant because the sum non-diagonal entries of $M_1$ is constant and $Q$ has constant row sums. Therefore $M_1$ has constant diagonal entries and moreover constant row sums.

Thus by Lemma \ref{lem:similarity}, $\mathcal{S}$ is a similarity matrix for $N(\lambda,r)$ of $G_1,G_2$. Therefore $\phi_{G_1}(\lambda,r)=\phi_{G_2}(\lambda,r)$. 
\end{proof}

This allows us to state our construction for the normalized Laplacian matrix since the characteristic polynomial of the normalized Laplacian can be written in terms of the generalized characteristic polynomial. 

\begin{cor} \label{cor:normLapla}
Let $G$ be a graph containing two vertex sets $V_1, V_2$ each on $m$ vertices such that 

\begin{enumerate}
    \item $G[V_1],G[V_2]$ are empty subgraphs;
    \item there exists a graph automorphism $\pi$ for $G[V_1 \cup V_2]$ such that $\pi(V_1)=V_2$, $\pi(V_2)=V_1$, and $\pi^2=id$. 
\end{enumerate}
Let $H_1$ and $H_2$ be any two graphs on $m$ vertices and $\phi_i$ be a bijective mapping from $H_i$ to $V_i$ for $i \in \{1,2\}$. Let $G_1=G+ \phi_{1}(E(H_1)) + \phi_{2}(E(H_2)))$  and $G_2=G+ \pi(\phi_{1}(E(H_1))) + \pi(\phi_{2}(E(H_2))))$. 

If $V_1$ is a set of co-degree cousins to $V_2$ and $G_1[V_1 \cup V_2]$ is a regular graph, then $G_1$ and $G_2$ are cospectral for the Normalized Laplacian.
\end{cor}

\begin{proof}
This follows immediately from Corollary \ref{cor:genCharPoly}. 
\end{proof}

Figure \ref{fig:transmission regular} gives an example of a graph $G$ and graphs to glue in that meet the hypothesis of Corollary \ref{cor:genCharPoly}, therefore the graphs given in Figure \ref{fig:transmission regular} (b) and (c) have the same generalized characteristic polynomial and are cospectral for the normalized Laplacian. 

\section{Concluding Remarks}

We have presented an extension of a construction method and applied it to many matrices. 
A natural question about this construction method is what fraction of cospectral graphs does this explain for each matrix? In addition, there are smaller examples then those shown in this paper for some matrices, but it is unknown if there is a smaller example for the distance matrix. Is there an example of a pair of graphs that are cospectral for all six matrices discussed here? 

There is also some evidence that the graph automorphism $\pi$ as described in Theorem \ref{cousinConst} does not need $\pi(V_1)=V_2$ and $\pi(V_2)=V_1$ for a similar construction shown in \cite{grwc}. Finding other conditions or cases when $\pi$ is some other graph automorphism where the spirit of Theorem \ref{cousinConst} holds is an interesting open problem. 

Cospectral constructions have now been shown for adjacency matrices using $\diag(\frac{1}{k}J -I, I)$ (\cite{GM,HS04}) and $\diag(\frac{1}{k}J - \hat{I},I)$ (Theorem \ref{cousinConst}) as similarity matrices. Since both $I$ and $\hat{I}$ are symmetric permutation matrices, is there a cospectral construction for every symmetric permutation matrix $P$ where $\diag(\frac{1}{k}J-P,I)$ is the similarity matrix? This problem has been studied for orthogonal matrices in \cite{AH12, WQH19} where the construction method is switching.


\begin{thebibliography}{99}

\bibitem{AH12} A. Abiad and W. H. Haemers. Cospectral graphs and regular orthogonal matrices of level 2. \emph{Elec J. of Combinatorics.} 19(3) (2012) P13, 16pp.

\bibitem{AH13} M. Aouchiche and P. Hansen, Two Laplacians for the distance matrix of a graph. \emph{Linear Algebra Appl.} 439 (2013). 21--33.

\bibitem{grwc} B. Brimkov, K. Duna, L. Hogben, K. Lorenzen, C. Reinhart, S.-Y. Song, and M. Yarrow. Graphs that are cospectral for the distance Laplacian. Preprint (2018) 	arXiv:1812.05734

\bibitem{BG11} S. Butler and J. Grout. A construction of cospectral graphs for the normalized Laplacian. \emph{Elec J. of Combinatorics.} 18 (2011) P231, 20pp.  

\bibitem{das} K. Ch. Das. The Laplacian spectrum of a graph. \emph{Computers and Mathematics with App.} 48 (2004). 715--724.

\bibitem{vDH03} E. R. van Dam and  W. H. Haemers. Which graphs are determined by their spectrum? \emph{Linear Algebra Appl.} 373 (2003). 139--162.

\bibitem{GM} C. Godsil and B. D. McKay. Constructing cospectral graphs. \emph{Aeq. Math.} 25 (1982) 257-268. 

\bibitem{HS04} W.H. Haemers and E. Spence. Enumeration of cospectral graphs. \emph{European Journal of Combinatorics.} 25 (2004). 199--211.

\bibitem{heysse} K. Heysse. A construction of distance cospectral graphs. \emph{Linear Algebra and its Applications.} 535 (2017). 195--212.


\bibitem{merris} R. Merris. Laplacian matrices of graphs: a survey. \emph{Linear Algebra and its Applications.} 197 (1994). 143--176.


\bibitem{WQH19} W. Wang, L. Qiu, and Y. Hu. Cospectral graphs, GM-switching and regular rational orthogonal matrices of level p. \emph{Linear Algebra and its Applications.} 563 (2019). 154--177. 

\end{thebibliography}
\end{document}